\documentclass[12pt]{article}
\usepackage{amsmath}
\usepackage{amsthm}
\usepackage{amssymb}
\usepackage{fullpage}
\usepackage[colorlinks=true,citecolor=black,linkcolor=black,urlcolor=blue]{hyperref}
\allowdisplaybreaks
\usepackage{graphicx}

\newtheorem{thm}{Theorem}[section]
\newtheorem{lemma}[thm]{Lemma}
\newtheorem{prop}[thm]{Proposition}
\newtheorem{claim}[thm]{Claim}
\newtheorem{cor}[thm]{Corollary}

\theoremstyle{definition}
\newtheorem*{defn}{Definition}

\theoremstyle{remark}
\newtheorem*{obs}{Observation}

\DeclareMathAlphabet\mathbfcal{OMS}{cmsy}{b}{n}
\def\b{\boldsymbol}
\def\e{\mathrm}
\def\be{\mathbf}
\def\Proj{\mathrm{Proj}}
\def\bProj{\mathbf{Proj}}
\def\G{\mathbfcal G}
\def\E{\mathbfcal E}
\def\J{\mathrm J} 
\def\bJ{\mathrm J}

\title{Asymptotics of Pattern Avoidance in the Klazar Set Partition and Permutation-Tuple Settings}
\author{Benjamin Gunby \\ 
\small Department of Mathematics\\[-0.8ex]
\small Harvard University\\[-0.8ex]
\small Cambridge, Massachusetts, U.S.A.\\
\small{\tt bgunby@g.harvard.edu}
\and D\"om\"ot\"or P\'alv\"olgyi\footnote{Research supported by the Marie Sk\l odowska-Curie action of the EU, under grant IF 660400.} \\
\small Dep.\ of Pure Mathematics and Mathematical Statistics\\[-0.8ex]
\small University of Cambridge\\[-0.8ex]
\small Cambridge, UK\\
\small{\tt dom@cs.elte.hu}
}

\begin{document}
\maketitle

\begin{abstract}
We consider asymptotics of set partition pattern avoidance in the sense of Klazar. Our main result derives the asymptotics of the number of set partitions avoiding a given set partition within an exponential factor, which leads to a classification of possible growth rates of set partition pattern classes. We further define a notion of permutation-tuple avoidance, which generalizes notions of Aldred et al. and the usual permutation pattern setting, and similarly determine the number of permutation-tuples avoiding a given tuple to within an exponential factor.
\end{abstract}

\section{Introduction}
A fundamental question of pattern avoidance is that of asymptotics. That is, for some pattern $p$, how does the avoidance function $A_n(p)$, equal to the number of patterns of size $n$ avoiding $p$, grow? More generally, what are the possible growth speeds of pattern classes? This has been especially well-studied in the most classical pattern avoidance area, that of permutations. The most famous result of this kind is the Marcus-Tardos Theorem, known earlier as the Stanley-Wilf Conjecture \cite{StanleyWilf}, and generalized repeatedly in later years.

Recently, the study of pattern avoidance and the corresponding asymptotics has branched into other structures than permutations. Klazar \cite{K} proposed a notion of set partition pattern avoidance. (There are other possible definitions of set partition avoidance; for example, RGF-type pattern avoidance, studied in great detail by for example Mansour \cite{M}, which we will not discuss in this paper.) Klazar proved several results about special cases involving the generating function of the avoidance sequence. Later, his conjecture about the case when the partitions have what this paper will refer to as permutability $1$ (Klazar refers to these as srps) was proven independently by Klazar and Marcus \cite{KM} and Balogh, Bollob\'as, and Morris \cite{BBM}. The primary result of this paper will be a generalization of those results, completely classifying speeds of set partition pattern classes up to an exponential factor.

We will see that the study of set partition pattern avoidance motivates a new notion of $k$-tuple permutation pattern avoidance, which generalizes classical permutation pattern avoidance (corresponding to $k=1$) and the setting of pairs of permutations in Aldred et al. \cite{AADHHM}, which corresponds to $k=2$. We will see that our set partition result implies a generalization of the Marcus-Tardos theorem to this setting.
\section{Definitions and Preliminary Results}
\begin{defn}
A \emph{set partition} of $n$ is a partition of $[n]$ into sets, where we ignore ordering of sets and ordering within the sets. We will write set partitions with slashes between the sets, as in $T_1/T_2/\cdots/T_m$ for some $m$. The \emph{standard form} of a set partition is what is obtained from writing each $T_i$ in increasing order, and then rearranging the sets so that $\min T_1<\min T_2<\cdots<\min T_m$. The $T_i$ are called the \emph{blocks} of the partition.
\end{defn}
For example, $1635/24$ and $24/1356$ are not in standard form; the standard form for this partition is $1356/24$.

\begin{defn}
The \emph{Bell number} $B_n$ is the number of set partitions of $[n]$.
\end{defn}

\begin{defn}
A set partition $\pi$ of $n$ \emph{contains} a set partition $\pi'$ of $k$ in the Klazar sense (which we will use for the remainder of this paper) if there is a subset $S$ of $[n]$ of cardinality $k$ such that when $\pi$ is restricted to the elements of $S$, the result is order-isomorphic to $\pi'$. Otherwise, we say $\pi$ \emph{avoids} $\pi'$.
\end{defn}
For example, $136/5/27$ contains $14/23$ because when we restrict $136/5/27$ to the set $\{2,3,6,7\}$, the result is $36/27$, which is order-isomorphic to $23/14$, standardizing to $14/23$. However, it avoids $1/234$.

We can think of containment in the following way: if we have some $f:[m]\to[n]$ and a set partition of $[n]$, we can take the pullback under $f$ to get a partition of $[m]$, where $a$ and $b$ are in the same partition if and only if $f(a)$ and $f(b)$ are. Then $\pi$ contains $\pi'$ if and only if $\pi'$ is the pullback of $\pi$ under some order-preserving injection.

Note that this Klazar notion of avoidance differs from the RGF notion of pattern avoidance in set partitions, studied in detail by Mansour \cite{M}, where switching the order of the sets during standardization is not allowed.

We will be concerned with the enumeration of the number of partitions of a given length that avoid a particular pattern.
\begin{defn}
If $\pi$ is some set partition of $k$, let $B_n(\pi)$ be the number of set partitions of $n$ that avoid $\pi$. (Note that the notation is analogous to that for Bell numbers.)
\end{defn}
Much of this paper is devoted to progress towards general asymptotic bounds for $B_n(\pi)$.
\begin{defn}
A \emph{layered partition} is a partition $T_1/\cdots/T_m$ such that $\max T_i<\min T_{i+1}$ for all $i\in [m-1]$. Equivalently, each set consists of an interval of consecutive integers.
\end{defn}
For example, $12/3456/789$ is layered while $13/2456/789$ is not.

Alweiss \cite{Ryan} found the correct log-asymptotic for $B_n(\pi)$ in the case where $\pi$ is layered. Earlier, Klazar and Marcus in \cite{KM} classified the cases where $B_n(\pi)$ grows at most exponentially, as a corollary of their Corollary 2.2.

An important notion in this paper will be relating set-partition pattern avoidance to tuple permutation pattern avoidance. To this end, we define the following notion.
\begin{defn}
Let $\sigma_1,\ldots,\sigma_d$ be permutations $[n]\to [n]$. We define the \emph{set partition correspondent to $(\sigma_1,\ldots,\sigma_d)$} to be the partition $T_1/\cdots/T_n$ of $(d+1)n$ such that $T_i=\{i,n+\sigma_1(i),2n+\sigma_2(i),\ldots,dn+\sigma_d(i)\}$. It is easy to see that this is indeed a set partition, and we will write it $[\sigma_1,\ldots,\sigma_d]$.
\end{defn} 
Notice that a set partition of $(d+1)n$ is correspondent to some $(\sigma_1,\ldots,\sigma_d)$ if and only if every set in the partition contains exactly one element from each of $\{1,\ldots,n\}$, $\{n+1,\ldots,2n\}$,\ldots,$\{dn+1,\ldots,(d+1)n\}$. Klazar \cite{K} referred to partitions of the form $[\sigma]$ (so $d=1$) or partitions contained in any partition of this form as \emph{srp}'s, and as previously alluded to, Klazar and Marcus proved in \cite{KM} that for $\pi$ an srp, there exists $c>0$ with $B_n(\pi)\leq c^n$.

Now, we define what we will call \emph{parallel pattern avoidance} for $d$-tuples of permutations $(\sigma_1,\ldots,\sigma_d)$.
\begin{defn}
If $\sigma_1,\ldots,\sigma_d\in S_n$ and $\sigma'_1,\ldots,\sigma'_d\in S_m$, $(\sigma_1,\ldots,\sigma_d)$ contains (respectively avoids) $(\sigma'_1,\ldots,\sigma'_d)$ if there exists (respectively does not exist) indices $c_1<\cdots<c_m$ such that $\sigma_i(c_1)\sigma_i(c_2)\cdots\sigma_i(c_m)$ is order-isomorphic to $\sigma'_i$ for all $i$.
\end{defn}
We will occasionally say `contains/avoids in parallel' to refer to this notion in particular.

For $d=1$, parallel pattern avoidance is equivalent to the classical case of permutation pattern containment/avoidance. This idea of parallel avoidance in $d$-tuples of permutations also reduces to several other interesting concepts in special cases; for example, $(\sigma_1,\sigma_2)$ avoids $(12,21)$ if and only if $\sigma_1^{-1}\leq\sigma_2^{-1}$ in the Weak Bruhat Order, which has been previously studied; for example, see \cite{HP} and $A007767$ in \cite{OEIS}.

We now relate this to our topic of partition pattern avoidance.
\begin{prop}\label{parallel}
Let $\sigma_1,\ldots,\sigma_d$ be permutations in $S_n$ and $\sigma'_1,\ldots,\sigma'_d$ be permutations in $S_m$. The following two statements are equivalent:
\begin{itemize}
\item
The $d$-tuple of permutations $(\sigma_1,\ldots,\sigma_d)$ contains the $d$-tuple of permutations $\sigma'_1,\ldots,\sigma'_d$.
\item
The set partition $[\sigma_1,\ldots,\sigma_d]$ contains the set partition $[\sigma'_1,\ldots,\sigma'_d]$.
\end{itemize}
\end{prop}
\begin{proof}
If $(\sigma_1,\ldots,\sigma_d)$ contains $(\sigma'_1,\ldots,\sigma'_d)$, we have indices $c_1<\cdots<c_m$ with $\sigma_i(c_1)\cdots\sigma_i(c_m)$ order-isomorphic to $\sigma_i$. $[\sigma_1,\ldots,\sigma_d]$ has blocks $T_1,\ldots,T_n$ given by $T_i=\{i,n+\sigma_1(i),2n+\sigma_2(i),\ldots,dn+\sigma_d(i)\}$. Restricting this to simply the elements in $T_{c_1},\ldots,T_{c_m}$, we have blocks given by $\{c_i,n+\sigma_1(c_i),\ldots,dn+\sigma_d(c_i)\}$. We show that this is order-isomorphic to $[\sigma'_1,\ldots,\sigma'_d]$. Since $c_i=\min T_{c_i}$, and the $c_i$ are increasing, the block $T_{c_i}$ must correspond the $i^{th}$ block of $[\sigma'_1,\ldots,\sigma'_d]$, which is $\{i,m+\sigma'_1(i),\ldots,dm+\sigma'_d(i)\}$. Thus, we must show that $j_1n+\sigma_{j_1}(c_{i_1})<j_2n+\sigma_{j_2}(c_{i_2})$ if and only if $j_1m+\sigma_{j_1}(i_1)<j_2m+\sigma_{j_2}(i_2)$. But since $1\leq\sigma_a(b)\leq n$ and $1\leq\sigma'_a(b)\leq m$ for all $a,b$, the first statement is equivalent to $j_1<j_2$ or $j_1=j_2=j$ and $\sigma_j(c_{i_1})<\sigma_j(c_{i_2})$, and the second is equivalent to $j_1<j_2$ or $j_1=j_2=j$ and $\sigma'_j(i_1)<\sigma'_j(i_2)$. These are equivalent by the definition of pattern containment for $k$-tuples of permutations.

Now suppose $[\sigma_1,\ldots,\sigma_d]$ contains $[\sigma'_1,\ldots,\sigma'_d]$. Since all blocks of both partitions have size $d+1$, the blocks of the latter partition must correspond exactly to $m$ block of the former, say blocks $T_{c_1},\ldots,T_{c_m}$ with $c_1<\cdots<c_m$. Now following the exact same argument in reverse, we see that $(\sigma_1,\ldots,\sigma_d)$ contains $(\sigma'_1,\ldots,\sigma'_d)$ (at indices $c_1,\ldots,c_m$), as we showed the ordering information is exactly equivalent in both cases.
\end{proof}
The concept of permutation-correspondent partitions gives us a useful statistic.
\begin{defn}
The \emph{permutability} of a set partition $\pi$, which we will call $\text{pm}(\pi)$, is the minimum $d$ such that there exists a $d$-tuple of permutations $(\sigma_1,\ldots,\sigma_d)$ such that the correspondent partition $[\sigma_1,\ldots,\sigma_d]$ contains $\pi$.
\end{defn}
Note that as one would expect, $[\sigma_1,\ldots,\sigma_d]$ has permutability $d$, as it has a block of size $d+1$, which is not contained in $[\sigma'_1,\ldots,\sigma'_{d-1}]$ for any choice of the $\sigma'_i$.
If $\pi$ is the set partition of $n$ with all blocks of size $1$, then $\text{pm}(\pi)=0.$

\section{Old and New Results}
A main purpose of this paper is to determine as closely as possible the asymptotics of $B_n(\pi)$. It is not difficult to show a lower bound for $B_n(\pi)$; indeed, we will show the following.
\begin{thm}\label{lowerbound}
Let $\pi$ be a set partition with $\text{pm}(\pi)\geq 1$. Then there exists a constant $c_1(\pi)>0$ such that
\[B_n(\pi)\geq c_1(\pi)^nn^{n\left(1-\frac{1}{\text{pm}(\pi)}\right)}\]
for all $n$.
\end{thm}
We will also prove the following upper bound, which will determine the growth rate of $B_n(\pi)$ to within an exponential factor.
\begin{thm}\label{mainthm}
Let $\pi$ be a set partition with $\text{pm}(\pi)\geq 1$. Then there exists a constant $c_2(\pi)$ such that
\[B_n(\pi)\leq c_2(\pi)^nn^{n\left(1-\frac{1}{\text{pm}(\pi)}\right)}\]
for all $n$. If $\text{pm}(\pi)=0$, then there exists a constant $c_2(\pi)$ so that
\[B_n(\pi)\leq c_2(\pi)^n.\]
\end{thm}
Note that Klazar and Marcus proved Theorem \ref{mainthm} in the case where $\text{pm}(\pi)=1$ in \cite{KM}.

The most general result of this paper deals with asymptotics of parallel avoidance. We first give the following definition.
\begin{defn}
If $\sigma_1,\ldots,\sigma_d$ are permutations of some $[m]$, we say that $S_n^d(\sigma_1,\ldots,\sigma_d)$ is the number of $d$-tuples of permutations $(\sigma'_1,\ldots,\sigma'_d)$ with $\sigma'_i\in S_n$ such that $(\sigma'_1,\ldots,\sigma'_d)$ avoids $(\sigma_1,\ldots,\sigma_d)$ in parallel.
\end{defn}

The famous Marcus-Tardos Theorem \cite{StanleyWilf}, building on the work of Klazar \cite{K3}, states the following (corresponding to the case $d=1$).

\begin{thm}[Marcus-Tardos \cite{StanleyWilf}]\label{StanleyWilf}
Let $m\in\mathbb{N}$. For any permutation $\sigma\in S_m$, let $S_n(\sigma)=S_n^1(\sigma)$ be the number of permutations in $S_n$ avoiding $\sigma$. Then for all $\sigma$ there exists a constant $c$ such that
\[S_n(\sigma)\leq c^n.\]
\end{thm}

Let $\sigma_1,\ldots,\sigma_d$ be permutations, say in $S_m$. Then for every $(\sigma'_1,\ldots,\sigma'_d)$ that avoids $(\sigma_1,\ldots,\sigma_d)$, we have a corresponding set partition $[\sigma'_1,\ldots,\sigma'_d]$ avoiding $[\sigma_1,\ldots,\sigma_d]$ by Proposition \ref{parallel}. Thus, Theorem \ref{mainthm} should imply a corresponding bound on parallel permutation pattern avoidance. This turns out to suggest a natural generalization of Theorem \ref{StanleyWilf} to $d$-tuples, in the form of the following.

\begin{thm}\label{permthm}
Let $m>1$ and let $\sigma_1,\ldots,\sigma_d\in S_m$ be permutations. Then the following hold.
There exists constants $c_2>c_1>0$ (depending on the $\sigma_i$) such that $c_1^nn^{n\frac{d^2-1}{d}}\leq S_n^k(\sigma_1,\ldots,\sigma_k)\leq c_2^nn^{n\frac{d^2-1}{d}}$ for all $n$.
\end{thm}

\section{Proof of Theorem \ref{lowerbound}}
We will now prove Theorem \ref{lowerbound}.

Let $\pi$ be a set partition with $\text{pm}(\pi)=d$. Assume $d>1$, as the case $d=1$ is trivial. By the interval criterion for permutability, removing blocks containing one element from $\pi$ does not change its permutability (as it preserves intervals containing exactly one element from each set). Thus, if $\pi'$ is $\pi$ with all one-element blocks removed, any partition avoiding $\pi'$ must avoid $\pi$ since $\pi$ contains $\pi'$, so $B_n(\pi)\geq B_n(\pi')$ and $\text{pm}(\pi')=d$. So it suffices to show the problem for $\pi'$; that is, we can assume without loss of generality that $\pi$ has no blocks of size $1$. This means that we can add any blocks of size $1$ to a partition of $[n-i]$ avoiding $\pi$ to get a partition of $[n]$ avoiding $\pi$. If we only range over partitions of $[n-i]$ with no blocks of size $1$, the resulting partitions will all be distinct. Let $B'_n([\pi])$ be the number of partitions of $[n]$ avoiding $[\pi]$ with no blocks of size $1$. Then since we can perform the process of adding single blocks in $\binom{n}{i}$ ways, we have $B_n(\pi)\geq\binom{n}{i}B'_{n-i}(\pi)$.

Now suppose $n$ is a multiple of $d$, $n=dm$. Then if $\sigma_1,\ldots,\sigma_{d-1}\in S_m$ are permutations, then $[\sigma_1,\ldots,\sigma_{d-1}]$ will be a partition of $[m(d-1+1)]=[n]$, and by the definition of permutability, it must avoid $\pi$. Since these all correspond to different partitions, and all blocks have size $d>1$, we can count them to see that
\[B'_n(\pi)\geq m!^{d-1}=\left(\frac{n}{d}\right)!^{d-1}.\]
By Stirling Approximation, there is $c>0$ such that $\left(\frac{n}{d}\right)!>c^{\frac{n}{d}}\left(\frac{n}{d}\right)^{\frac{n}{d}}=\left(\frac{c}{d}\right)^{\frac{n}{d}}n^{\frac{n}{d}}$. Substituting this in,
\[B'_n(\pi)\geq\left(\frac{c}{d}\right)^{\frac{(d-1)n}{d}}n^{\frac{(d-1)n}{d}}=c_0^nn^{n\left(1-\frac{1}{d}\right)},\]
where $c_0=\left(\frac{c}{d}\right)^{\frac{d-1}{d}}$.

Now we use this to solve the case where $d\nmid n$. Let $n=dm+i$, $0\leq i\leq d-1$. Since we are dealing with asymptotics we may assume that $n>d$. We have that since $n-i$ is a multiple of $d$, assuming $c_0<1$ without loss of generality for ease of manipulation,
\begin{align*}
B_n(\pi) & \geq\binom{n}{i}B'_{n-i}(\pi) \\ & \geq\binom{n}{i}c_0^{n-i}(n-i)^{(n-i)\left(1-\frac{1}{d}\right)} \\ & \geq\frac{(n-i)^i}{i!}c_0^{n-i}(n-i)^{(n-i)\left(1-\frac{1}{d}\right)} \\ & =\frac{c_0^n}{c_0^ii!}(n-i)^{(n-i)\left(1-\frac{1}{d}\right)+i} \\ & =\frac{c_0^n}{c_0^ii!}(n-i)^{n\left(1-\frac{1}{d}\right)+\frac{i}{d}} \\ & \geq\frac{c_0^n}{c_0^ii!}(n-i)^{n\left(1-\frac{1}{d}\right)} \\ & =\frac{c_0^n}{c_0^ii!}\left(1-\frac{i}{n}\right)^{n\left(1-\frac{1}{d}\right)}n^{n\left(1-\frac{1}{d}\right)} \\ & >\frac{c_0^n}{d!}\left(1-\frac{d}{n}\right)^{n\left(1-\frac{1}{d}\right)}n^{n\left(1-\frac{1}{d}\right)}.
\end{align*}
Since $\left(1-\frac{d}{n}\right)^n$ is positive for $n\in [k+1,\infty]$ and limits to $e^{-k}\neq 0$ as $n\to\infty$, it must have a minimum, call it $a$, on $n\in [d+1,\infty]$. Substituting this in and noting $a<1$,
\begin{align*}
B_n(\pi) & >\frac{c_0^n}{d!}\left(1-\frac{d}{n}\right)^{n\left(1-\frac{1}{d}\right)}n^{n\left(1-\frac{1}{d}\right)} \\ & =\frac{c_0^n}{d!}\left(\left(1-\frac{d}{n}\right)^n\right)^{\left(1-\frac{1}{d}\right)}n^{n\left(1-\frac{1}{d}\right)} \\ & \geq\frac{c_0^n}{d!}a^{\left(1-\frac{1}{d}\right)}n^{n\left(1-\frac{1}{d}\right)} \\ & \geq\frac{c_0^na}{d!}n^{n\left(1-\frac{1}{d}\right)} \\ & >\left(\frac{c_0a}{d!}\right)^nn^{n\left(1-\frac{1}{d}\right)} \\ & =c_1^nn^{n\left(1-\frac{1}{d}\right)},
\end{align*}
where $c_1=\frac{c_0a}{k!}$. This concludes the proof of the theorem.

\section{Ordered Hypergraph Pattern Avoidance}
We start this section by defining ordered hypergraph pattern avoidance.
\begin{defn}
Let $G$ and $H$ be hypergraphs whose vertex sets are totally ordered. Then $G$ \emph{contains} $H$ if there exists both an order-preserving injection $V(H)\to V(G)$ and an injection $E(H)\to E(G)$ such that the two are compatible--that is, if $E\in E(H)$ is sent to $E'\in E(G)$, then every vertex of $E$ is sent to a vertex of $E'$ under the map of vertices (note that this map $V(E)\to V(E')$ need not be surjective). If $G$ does not contain $H$, we as usual say that $G$ \emph{avoids} $H$.
\end{defn}
\begin{defn}
The \emph{weight} of a hypergraph $G$, denoted $i(G)$, is the sum of the sizes of all edges in $G$, $\displaystyle\sum_{E\text{ an edge of G}}|E|$. We will denote by $e(G)$ the number of edges in $G$.
\end{defn}
We also define a $d$-permutation hypergraph.
\begin{defn}
A \emph{$d$-permutation hypergraph} is a hypergraph $G$ on the vertex set $[kd]$ for some $k\in\mathbb{Z}^+$, such that the following properties are satisfied.
\begin{itemize}
\item
$G$ has $k$ edges, each of size $d$, such that each vertex is in exactly one edge.
\item
Each edge has exactly one vertex from each of $\{1,\ldots,k\}$, $\{k+1,\ldots,2k\}$, \ldots, and $\{(d-1)k+1,\ldots,dk\}$.
\end{itemize}
\end{defn}
In Section 2 of \cite{KM} and independently as Lemma $14$ of \cite{BBM}, the following generalization of the F\"{u}redi-Hajnal conjecture \cite{FH} (which occurs when $G$ is bipartite and was proved by Marcus and Tardos \cite{StanleyWilf}) was proven.
\begin{thm}\label{thm:km}[Balogh-Bollob\'as-Morris \cite{BBM}, Klazar-Marcus \cite{KM}]
Let $H$ be a fixed $2$-permutation hypergraph. Then for any $n\in\mathbb{Z}^+$ and hypergraph $G$ on $[n]$ avoiding $H$, $i(G)=O(n)$.
\end{thm}
This was a key lemma in the proof of the $\text{pm}(\pi)=1$ case of Theorem \ref{mainthm}. We prove the following generalization of this result to deduce Theorem \ref{mainthm} from it.
\begin{thm}\label{hypthm}
Let $H$ be a fixed $d$-permutation hypergraph. Then for any $n\in\mathbb{Z}^+$ and hypergraph $G$ on $[n]$ avoiding $H$, $i(G)=O(n^{d-1})$.
\end{thm}
Our proof most resembles the respective proof in \cite{MP} but very likely the methods of \cite{BBM} and \cite{KM} could be modified in a similar fashion.

\section{Proof of Theorem \ref{hypthm}}
We will first show Theorem \ref{hypthm} in the case where $G$ is $t$-uniform for a fixed $t$. In fact, we will prove something stronger by induction. First, we need to define the {\em projection} of a $t$-uniform hypergraph.
\begin{defn}
Let $G$ be a $t$-uniform ordered hypergraph, and let $\J$ be a subset of $[t]$ of cardinality $a$. For an edge $\e E\in G$, let $\Proj_{\J} \e E$ be the hyperedge of cardinality $t-a$ given by deleting the $i^{th}$ vertex of $\e E$ for all $i\in {\J}$. Let $\Proj_{\J} G$ be the $(t-a)$-uniform hypergraph given by the same vertex set as $G$ and the edges $\Proj_{\J} \e E$ for all edges $\e E\in G$ (only counting multiple edges once).
\end{defn}

\begin{obs}\label{obs:proj} If $G$ is $t$-uniform, $G$ contains $\Proj_{\J} G$ for any ${\J}\subset [t]$.
\end{obs}

Our proof is quite long and uses several projections; this makes it sometimes quite confusing to recall whether a hypergraph is $t$-uniform, $d$-uniform or $(d-1)$-uniform etc.
Because of this, throughout the statement and proof of the next lemma, we use the following notational conventions.
We denote $t$-uniform hypergraphs with a bold letter $\b G$ (and possible further indices), $(t-1)$-uniforms with a normal letter $G$, $d$-uniform hypergraphs with a letter $\b H$ (or, at a later part, one will be $\G$) and $(d-1)$-uniform hypergraphs with a letter $H$.
Similar rules are used for hyperedges ($\be E$ means size $t$ or $d$, while $\e E$ means size $t-1$ or $d-1$), %($\bJ$ means size $t-d$ and $\J$ means size $t-d-1$), %unfortunately, depending on the proof, it changes which two we use -sometimes it's $t-d+1$ and $t-d$
and projections; so $\bProj_{\J} G$ would project to a $d$-uniform hypergraph, while $\Proj_{\bJ} G$ to a $(d-1)$-uniform hypergraph from the same $G$.

Our strengthening of Theorem \ref{hypthm} for $t$-uniform hypergraphs is the following.
\begin{lemma}\label{tuniform}
Fix $t,d,k\in\mathbb{Z}^+$. Then there exists a constant $c_{t,d,k}$ such that for all $n$ and all $t$-uniform hypergraphs $\b G$ on $[n]$ with $e(\b G)>c_{t,d,k}n^{d-1}$, there exists $\bJ\subset [t]$ with $|\bJ|=t-d$ such that $\bProj_{\bJ} \b G$ contains every $d$-permutation hypergraph on $kd$ vertices.
\end{lemma}
If $\b G$ avoids $\b H$, then by the above Observation $\bProj_{\bJ} \b G$ must also avoid $\b H$ for all ${\bJ}$, and since $i(\b G)=t\cdot e(\b G)$, Lemma \ref{tuniform} is indeed a strengthening of Theorem \ref{hypthm} for $t$-uniform hypergraphs.
\begin{proof}[Proof of Lemma \ref{tuniform}]
The proof will be induction on $t,d,n$ (while $k$ is fix).

The base cases of $t<d$ or $d=1$ are simple. If $t<d$, then $e(\b G)\leq\binom{n}{t}\leq n^t\leq n^{d-1}$. If $d=1$, then, by the definition of avoidance, if the conclusion does not hold, $\bProj_{\bJ} \b G$ must have less than $k$ edges for all ${\bJ}$ of size $t-1$, where $k$ is the number of edges (which in this case just consist of a single vertex) of $\b H$. That is, for any $i\in [t]$, there are only $k-1$ choices for the $i^{th}$ vertex of the edges of $\b G$. Thus, $e(\b G)\leq (k-1)^t=O(1)$, as desired.

We now proceed to the inductive step.
Suppose that $\b G$ is a $t$-uniform hypergraph on vertex set $[n]$ that does not satisfy the conclusion of the lemma (that is, there is no ${\bJ}$ that satisfies the conditions of the lemma). We wish to show that $e(\b G)=O(n^{d-1})$.

Now, for some positive integer $s$ (which will be potentially large, but fixed independently of $n$), divide the vertices of $\b G$ (that is, the set $[n]$) up into {\em intervals} of size $s$, with the remainder in another interval (that is, our intervals are $\{1,\ldots,s\},\{s+1,\ldots,2s\},\ldots,\{\left(\left\lfloor\frac{n}{s}\right\rfloor-1\right)s+1,\ldots,\left\lfloor\frac{n}{s}\right\rfloor s\},\{\left\lfloor\frac{n}{s}\right\rfloor s+1,\ldots,n\}$). Call these intervals $I_1,\ldots,I_{\left\lceil\frac{n}{s}\right\rceil}$.

Suppose $\be E$ is an edge which has at least two vertices in the same interval. Then let $f(\be E)$ be the smallest $i$ such that the $i^{th}$ and $(i+1)^{st}$ vertices of $\be E$ lie in the same interval. Let $\b G_0$ be the graph on $V(\b G)=[n]$ containing exactly the edges of $\b G$ which have at least two vertices in the same interval.

\begin{prop} $e(\b G_0)\leq c_{t-1,d,k}(t-1)(s-1)n^{d-1}=O(n^{d-1})$.
\end{prop}
\begin{proof}
Since $f(\be E)\in [t-1]$ for all edges $\be E\in \b G_0$, by the Pigeonhole Principle, at least $\frac{e(\b G_0)}{t-1}$ edges of $\b G_0$ must map to the same number under $f$. Let $\b G_1$ be a graph on $[n]$ with at least $\frac{e(\b G_0)}{t-1}$ edges all of which map to the same $i_0\in [t-1]$, i.e., $f(\be E)=i_0$ for all $\be E\in \b G_1$. Then for all $\be E\in \b G_1$, by definition, the $1^{st},\ldots,i_0^{th}$ elements of $\be E$ are in different intervals, and the $i_0^{th}$ and $(i_0+1)^{st}$ are in the same interval.

Consider the graph $G_2=\Proj_{\{i_0+1\}}\b G_1 $. Given any edge $\e E_2\in G_2$, it may correspond to multiple edges in $\b G_1 $. But if $\be E_1\in \b G_1 $ corresponds to $\e E_2\in G_2$, all of $\be E_1$'s vertices are determined, except for the $(i_0+1)^{st}$ vertex, which must be in the same interval as the $i_0^{th}$ (which is determined). Thus, there are at most $s-1$ choices for $\be E_1$ given $\e E_2$. So at most $s-1$ edges of $\b G_1 $ can correspond to any given edge of $G_2$, which implies that
\[e(G_2)\geq\frac{e(\b G_1 )}{s-1}\geq\frac{e(\b G_0)}{(t-1)(s-1)}.\]

Since $\b G_1 $ is obtained from $\b G$ by deleting some edges, $\Proj_{i_0+1}\b G$ contains $G_2=\Proj_{i_0+1}\b G_1$. If there is a ${\J}'\subset [t-1]$ with $|{\J}'|=t-d-1$ such that $\bProj_{\J'} G_2$ contains every $d$-permutation hypergraph on $kd$ vertices, then $\bProj_{\J'} \Proj_{i_0+1}\b G$ also contains every $d$-permutation hypergraph on $kd$ vertices. But the composition of two projections is itself a projection, in this case, by some $\bJ\subset [t]$ with $|\bJ|=t-d$. Thus, the conclusion of the lemma holds for $\b G$, contradicting our assumption.

Therefore, there is no $\J'\subset [t-1]$ with $|\J'|=t-d-1$ such that $\bProj_{\J'} G_2$ contains every $d$-permutation hypergraph on $kd$ vertices. By the inductive hypothesis (on $t$), there exists $c_{t-1,d,k}$ such that $e(G_2)\leq c_{t-1,d,k}n^{d-1}$. Thus, $e(\b G_0)\leq (t-1)(s-1)e(G_2)\leq c_{t-1,d,k}(t-1)(s-1)n^{d-1}$.
\end{proof}

Let $\b G'$ be the graph obtained from $\b G$ by removing the edges of $\b G_0$, thus, $\b G'$ contains the edges of $\b G$ all of whose vertices are in distinct intervals. We divide the edges of $\b G'$ into {\em blocks} depending on which intervals the vertices of each edge lie in; that is, $\be E$ and $\be E'$ are in the same block if and only if for all $i\in [t]$, the $i^{th}$ vertex of $\be E$ and $\be E'$ are in the same interval. Thus, there are $\left\lceil\frac{n}{s}\right\rceil^t$ possible blocks (some blocks may contain no edges).

Let $b$ be a block and $\b G_b$ be the graph with just the edges of that block; thus, $E(\b G')=\dot\cup_b E(\b G_b)$. For $\bJ\subset [t]$ with $|\J|=t-d+1$, we say that $b$ is {\em $\bJ$-wide} if $\Proj_{\bJ} \b G_b$  contains $H$ for every $(d-1)$-permutation hypergraph $H$ on $k(d-1)$ vertices. If there is no such $\bJ$, we say that $b$ is {\em thin} and, by the inductive hypothesis, there exists $c_{t,d-1,k}$ (not dependent on $s$) such that $e(\b G_b)\leq c_{t,d-1,k}s^{d-2}$.

Now we will bound the number of $\bJ$-wide blocks. Fix a particular ${\bJ}\subset t$ with $|{\bJ}|=t-d+1$. We partition the blocks into {\em ${\bJ}$-blockcolumns}; two blocks $b$ and $b'$ are in the same ${\bJ}$-blockcolumn if the intervals corresponding to the $i^{th}$ vertices of the edges in $b$ and $b'$ are the same for all $i\in [t]\setminus {\bJ}$. That is, a blockcolumn is obtained by fixing for every $i\in [t]\setminus {\bJ}$ which interval the $i^{th}$ vertex belongs to. Since $|[t]\setminus {\bJ}|=d-1$, there are at most $\binom{\left\lceil\frac{n}{s}\right\rceil}{d-1}$ ${\bJ}$-blockcolumns.

\begin{prop} Every ${\bJ}$-blockcolumn can have at most $(k-1)^{t-d+1}\binom{s}{k}^{(d-1)(k!)^{d-2}}$ blocks that are ${\bJ}$-wide.
\end{prop}
\begin{proof}
Suppose for the sake of contradiction that a particular ${\bJ}$-blockcolumn has at least $(k-1)^{t-d+1}\binom{s}{k}^{(d-1)(k!)^{d-2}}+1$ ${\bJ}$-wide blocks. For every $(d-1)$-permutation hypergraph $H$ on $k(d-1)$ vertices, we know that $H$ is contained in $\Proj_{\bJ} \b G_b$ for any of these blocks $b$. A copy of $H$ in $\Proj_{\bJ} \b G_b$ can occur in $\binom{s}{k}^{d-1}$ possible places (by a place we mean the injection of vertex sets given by containment), as for every $i\in [d-1]$ we must choose the $k$ locations in the corresponding interval that the $i^{th}$ vertices of the edges of $H$ are mapped to (there are $k$ such vertices). There are $(k!)^{d-2}$ such hypergraphs $H$, as for each set of $k$ vertices, except the first, we can match them up with the first $k$ by any permutation. Thus, there are $\binom{s}{k}^{(d-1)(k!)^{d-2}}$ ways that the copies of all the $(d-1)$-permutation hypergraphs on $k(d-1)$ vertices can occur in a block. Since two blocks in the same blockcolumn by the definition of blockcolumn have the same relevant intervals, by the Pigeonhole Principle, our blockcolumn must contain $(k-1)^{t-d+1}+1$ blocks where the copies of every $H$ occur on the exact same vertices.

Thus, again by the Pigeonhole Principle, there is some $i_0\in {\bJ}$ such that among these $(k-1)^{t-d+1}+1$ blocks, there are $k$ blocks such that the $i_0^{th}$ vertices of the edges of the $k$ blocks are in $k$ different intervals. Call these blocks $b_1,\ldots,b_k$, and assume they are sorted in increasing order of the interval the $i_0^{th}$ vertex is in.

\begin{claim} $\bProj_{{\bJ}\setminus \{i_0\}}\b G$ contains every $d$-permutation hypergraph $\b H$ on $kd$ vertices.
\end{claim}
\begin{proof}
Let $\G=\bProj_{\bJ\setminus \{i_0\}}\b G$, a $d$-uniform hypergraph.

Suppose that $i_0$ is the $j_0^{th}$ smallest element of $[n]\setminus \bJ$, so that $\Proj_{\bJ}$ and $\Proj_{j_0}\circ \bProj_{\bJ\setminus \{i_0\}}$ are the same operator. We now translate our conditions on $b_1,\ldots,b_k$ to conditions on blocks of $\G$. The blocks $b_1,\ldots,b_k$ will translate to blocks of $\G$, say $b'_1,\ldots,b'_k$. These blocks will have the property that for any $(d-1)$-permutation hypergraph $H'$ on $k(d-1)$ vertices, $\Proj_{j_0}\G_{b'_i}$ contains a copy of $H'$ for any $i$, moreover, these copies are located at exactly the same position for each $i$. Furthermore, the $j_0^{th}$ vertices of the blocks are all in different intervals.

In particular, $\Proj_{j_0}\G_{b'_i}$ contains a copy of $\Proj_{j_0}\b H$ in the same location for all $i$. We will use these copies to construct a copy of $\b H$ in $\G$.

Index the $k$ edges of $\b H$, $\be E_1,\ldots,\be E_k$, in increasing order of their $j_0^{th}$ vertex. Now, our copies of $\Proj_{j_0}\be H$ inside each $\Proj_{j_0}\G_{b'_i}$ give us compatible maps $E(\Proj_{j_0}\be H)\to E(\Proj_{j_0}\G_{b'_i})$ and $V(\Proj_{j_0}\b H)\to V(\Proj_{j_0}\G_{b'_i})$, where the second map is the same for all $i$ by our construction. Thus, for each edge $\be E_j\in \b H$, we can consider the edge it maps to in $\Proj_{j_0}\G_{b'_i}$, which in turn will be a projection of an edge in $\G$, which we denote by $\E_{i,j}$.

\begin{figure}%[h]
    \centering
		\includegraphics[width=\textwidth]{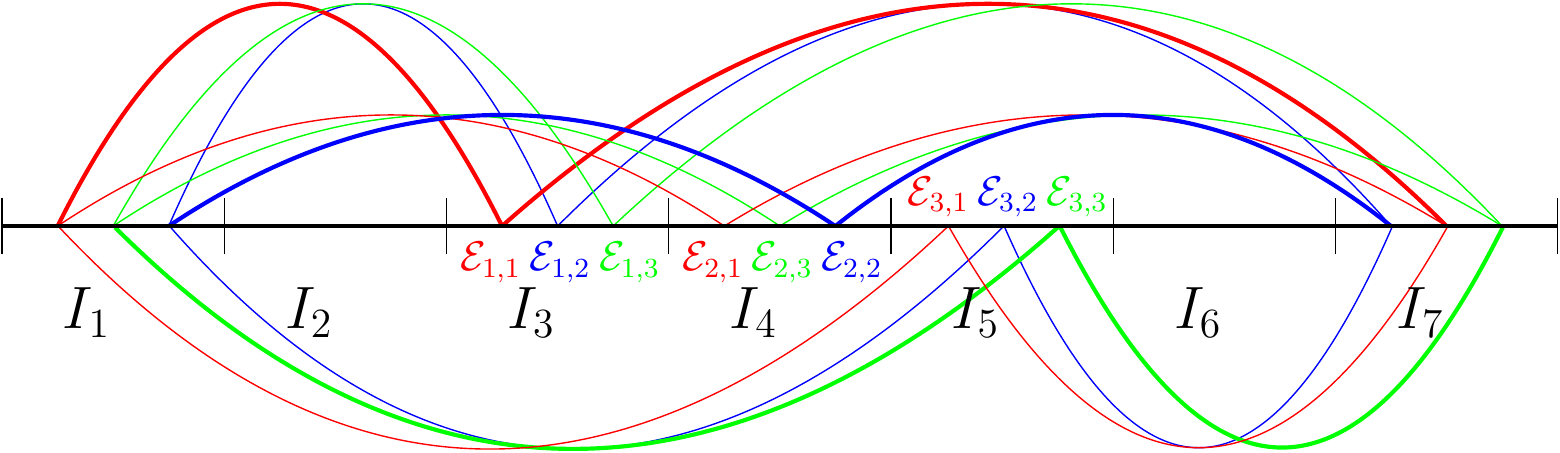}
	\caption{Example of position of edges $\E_{i,j}$. Color classes represent edges whose $\Proj_{j_0}$ image is the same (e.g., red is $\E_{1,1}$, $\E_{2,1}$, $\E_{3,1}$). Edges whose middle ($j_0^{th}$) vertex is in the same interval, belong to the same block (e.g., the middle vertex of $\E_{1,1}$, $\E_{1,2}$, $\E_{1,3}$ are all in $I_3$). The three bold edges from a $3$-permutation hypergraph $\b H$.}
	\label{fig:eij}
\end{figure}

By our construction, we know that the following hold (see Figure \ref{fig:eij}).

\begin{enumerate}
\item[(1)] $\Proj_{j_0}\E_{i,j}$ is independent of $i$, i.e., two edges, $\E_{i,j}$ and $\E_{i',j}$, differ only in their $j_0^{th}$ vertex. (This holds as the copies of $\Proj_{j_0}\b H$ occur in the same place in all blocks $i$.)
\item[(2)] For any $i$, the edges $\Proj_{j_0}\E_{i,j}$ (over all $j$) give us a copy of $\Proj_{j_0}\b H$, with $\Proj_{j_0}\E_{i,j}$ corresponding to edge $\be E_j$. %By (1) we may even choose a different $i_j$ for each $j$. ???
\item[(3)] For any fixed $i$, the $j_0^{th}$ vertices of $\E_{i,j}$, $v_{i1},\ldots,v_{ik}$, are in the same interval. These intervals ``increase'' with $i$, i.e., $v_{i,j}$ is in an earlier interval than $v_{(i+1),j'}$ for all $i,j,j'$.
\end{enumerate}

We now claim that $\E_{i,i}$, $1\leq i\leq k$, forms a copy of $\b H$ inside $\G$. We know that all, except possibly the $j_0^{th}$ vertices of the edges, are in the correct place by (1) and (2). All $j_0^{th}$ vertices are greater than all $(j_0-1)^{th}$ vertices and less than all $(j_0+1)^{th}$ vertices by (3). Finally, the $j_0^{th}$ vertices are in the correct order because the $v_{i,i}$ will be sorted in increasing order of $i$ by (3), and we chose the edges in $\b H$ to be sorted in increasing order as well. This proves the claim.
\end{proof}

Since $\G=\Proj_{\bJ\setminus \{i_0\}}G$, we have that $\Proj_{\bJ\setminus \{i_0\}}\b G$ contains $\b H$. Since $\b H$ was an arbitrary $d$-permutation hypergraph on $kd$ vertices (and $\bJ$ and $i_0$ were chosen independently of $\b H$), we have that $\Proj_{\bJ\setminus \{i_0\}}\b G$ contains all $d$-permutation hypergraphs on $kd$ vertices, which is a contradiction. Thus, our assumption must be false and every $J$-blockcolumn must have at most $(k-1)^{t-d+1}\binom{s}{k}^{(d-1)(k!)^{d-2}}$ blocks that are $\bJ$-wide, which finishes the proof of the proposition.
\end{proof}

Note that since (for a particular $\bJ$) the $\bJ$-blockcolumns are chosen by fixing $d-1$ distinct intervals in increasing order, and there are $\left\lceil\frac{n}{s}\right\rceil$ intervals, there are at most $\binom{\left\lceil\frac{n}{s}\right\rceil}{d-1}$ $\bJ$-blockcolumns. Thus, the total number of $\bJ$-wide blocks is at most $(k-1)^{t-d+1}\binom{s}{k}^{(d-1)(k!)^{d-2}}\binom{\left\lceil\frac{n}{s}\right\rceil}{d-1}$. Since there are $\binom{t}{d-1}$ choices for $\bJ$, the total number of blocks that are $\bJ$-wide for some choice of $\bJ$ is at most
\[\binom{t}{d-1}(k-1)^{t-d+1}\binom{s}{k}^{(d-1)(k!)^{d-2}}\binom{\left\lceil\frac{n}{s}\right\rceil}{d-1},\]
and thus, the number of edges in $\b G'$ in blocks that are not thin (i.e., $\bJ$-wide for some $\bJ$) is at most
\[s^t\binom{t}{d-1}(k-1)^{t-d+1}\binom{s}{k}^{(d-1)(k!)^{d-2}}\binom{\left\lceil\frac{n}{s}\right\rceil}{d-1}\]
(since each block may contain at most $s^t$ edges).

Now, we bound the number of nonempty thin blocks. Form a new ordered hypergraph $\b G_s$ from $\b G'$ in the following manner: $\b G_s$ will have $\left\lceil\frac{n}{s}\right\rceil$ vertices corresponding to the intervals in $[n]=V(\b G')$. The edges will correspond to nonempty thin blocks in the following manner: every nonempty block corresponds to a choice of $t$ intervals, in which the corresponding vertices of each edge of the block will reside. For each such nonempty thin block, we add a hyperedge to $\b G_s$ whose $t$ vertices will be the $t$ intervals corresponding to that block. So $\b G_s$ will also be $t$-uniform.

\begin{prop} $e(\b G_s)\le c_{t,d,k}\lceil\frac{n}{s}\rceil^{d-1}$.
\end{prop}
\begin{proof}
Using the induction hypothesis (on $n$), it is enough to show that there is no $\J\subset [t]$, $|\J|=t-d$, such that $\bProj_{\J} \b G_s$ contains all $d$-permutation hypergraphs $\b H$ on $kd$ vertices. Suppose the contrary. For each such $\b H$, this gives a set of $k$ edges in $\bProj_{\J} \b G_s$ (and thus, $k$ edges in $\b G_s$) that exhibit the containment. These correspond to $k$ edges of $\b G$, and since orders in $\b G_s$ are preserved in $\b G$, projecting these $k$ edges by $\J$ will also give a copy of $\b H$ in $\bProj_{\J} \b G$. Thus, $\bProj_{\J} \b G$ contains all $d$-permutation hypergraphs $\b H$ on $kd$ vertices, so $\bProj_{\J} \b G$ does as well, again a contradiction. %Thus, there does not exist such a $J$, so $\b G_s$ satisfies the contrapositive of the conclusion of Lemma \ref{tuniform}.
\end{proof}

We now put these parts together. %Let $f_{t,d,k}(n)$ be the maximum number of edges of a hypergraph $\b G$ like the one we chose (i.e., an ordered hypergraph on $[n]$ such that there does not exist a $\bJ$ that satisfies the conditions of Lemma \ref{tuniform}). For any such $\b G$,
We have shown the following:
\begin{enumerate}
\item $\b G$ has at most $c_{t-1,d,k}(t-1)(s-1)n^{d-1}$ edges with vertices in the same interval.
\item We may divide the remaining edges into blocks. There are at most
\[s^t\binom{t}{d-1}(k-1)^{t-d+1}\binom{s}{k}^{(d-1)(k!)^{d-2}}\binom{\left\lceil\frac{n}{s}\right\rceil}{d-1}\]
edges in non-thin blocks.
\item There are at most $c_{t,d,k}\lceil\frac{n}{s}\rceil^{d-1}$ nonempty thin blocks and each has at most $c_{t,d-1,k}s^{d-2}$ edges.
\end{enumerate}
Combining these, we obtain a bound.
\begin{align*}
|E(\b G)| & \leq c_{t,d-1,k}s^{d-2}c_{t,d,k}\left\lceil\frac{n}{s}\right\rceil^{d-1}+c_{t-1,d,k}(t-1)(s-1)n^{d-1} \\ & +s^t\binom{t}{d-1}(k-1)^{t-d+1}\binom{s}{k}^{(d-1)(k!)^{d-2}}\binom{\left\lceil\frac{n}{s}\right\rceil}{d-1} \\ & =c_{t,d-1,k}s^{d-2}c_{t,d,k}\left\lceil\frac{n}{s}\right\rceil^{d-1}+O\left(n^{d-1}\right),
\end{align*}
where the hidden constant in the $O$ notation does not depend on $c_{t,d,k}$.
Choosing the constant $s$ to be greater than $c_{t,d-1,k}$, the right hand side will be less than $c_{t,d,k}n^{d-1}$ for any sufficiently large constant $c_{t,d,k}$, completing the proof of Lemma \ref{tuniform}.
\end{proof}

We now use Lemma \ref{tuniform} to prove Theorem \ref{hypthm}.

Take some $d$-permutation hypergraph $H$ on $kd$ vertices, and let $G$ be a hypergraph on $[n]$ that avoids $H$. Again, break $[n]$ up into $\left\lceil\frac{n}{s}\right\rceil$ intervals of size at most $s$, for some $s$ (constant in $n$) that will be chosen later. Similarly to the proof of Lemma \ref{tuniform}, we form a new (multi)hypergraph $G_s$ on $\left[\left\lceil\frac{n}{s}\right\rceil\right]$, where the vertices correspond to the intervals. For each edge $E$ of $G$, we give $G_s$ an edge $E'$ so that an interval $I$ is a vertex of $E'$ if and only if $I$ contains a vertex of $E$. We then form a graph $G_s'$ by eliminating repeated edges of $G_s$. Note that, similarly to the proof of Lemma \ref{tuniform}, as $G$ avoids $H$, $G_s$ and $G_s'$ must as well.

Note that the edges of $G_s$ of size at least $kd$ can be repeated at most $k-1$ times, as $k$ copies of the same edge of size $kd$ would contain a copy of $H$. Now, edges of $G$ of size greater than $(kd-1)s$ must correspond to edges of size greater than $kd-1$ in $G_s$. We can split $i(G)$ into the contribution of edges of size at most $(kd-1)s$ and edges of size greater than $(kd-1)s$, say $i_<(G)$ and $i_>(G)$. The latter edges are repeated at most $k-1$ times and have their size reduced by a factor of at most $s$ when going from $G$ to $G_s'$, so $i_>(G)\leq (k-1)s\cdot i(G_s')$.

Now, for some $t$, the graph $G^t$ consisting of all size-$t$ edges of $G$ must also avoid $H$, so $\Proj_J G^t$ must avoid $H$ for all $J$ of size $t-d$. Thus, Lemma \ref{tuniform} implies that $e(G^t)\leq c_{t,d,k}n^{d-1}$. Summing up to $(kd-1)s$ and weighting by edge sizes, we see that $i_<(G)\leq\left(\displaystyle\sum_{t=1}^{(kd-1)s}tc_{t,d,k}\right)n^{d-1}$. Putting this together, we obtain that
\[i(G)\leq(k-1)s\cdot i(G_s')+\left(\displaystyle\sum_{t=1}^{(kd-1)s}tc_{t,d,k}\right)n^{d-1}.\]

Let $g(n)$ be the maximum value of $i(G)$ over all ordered hypergraphs $G$ on $[n]$ that avoid $H$. What we have shown above is that
\[g(n)\leq (k-1)s\cdot g\left(\left\lceil\frac{n}{s}\right\rceil\right)+\left(\displaystyle\sum_{t=1}^{(kd-1)s}tc_{t,d,k}\right)n^{d-1}=(k-1)s\cdot g\left(\left\lceil\frac{n}{s}\right\rceil\right)+O(n^{d-1}).\]
If $d>2$, we can choose $s>k-1$ and then the solution to this recurrence will be $O(n^{d-1})$, as desired. This just leaves the $d=2$ case, but this is simply Theorem \ref{thm:km}, finishing the proof.

\section{Proof of Theorem \ref{mainthm}}
We now will use Theorem \ref{hypthm} to prove Theorem \ref{mainthm}.

First note that the case $\text{pm}(\pi)=0$ is simple, as when $\pi=1/2/\cdots/k$, $\pi'$ avoids $\pi$ if and only if $\pi$ has at most $k-1$ blocks. Thus $B_n(\pi)\leq (k-1)^n$, so we may simply let $c_2$ equal $k-1$. For the remainder of the proof, we assume $\text{pm}(\pi)\geq 1$.

Note that if $\pi$ is contained in $\pi'$, then $B_n(\pi')\geq B_n(\pi)$. Thus, since every permutability-$d$ partition is contained in $[\sigma_1,\ldots,\sigma_d]$ for some permutations $\sigma_1,\ldots,\sigma_d$ (by definition of permutability), it suffices to show Theorem \ref{mainthm} in the case where $\pi=[\sigma_1,\ldots,\sigma_d]$.

Let $\sigma_1,\ldots,\sigma_d\in S_k$, $\pi=[\sigma_1,\ldots,\sigma_d]$ be the corresponding partition of $[(d+1)k]$, and $H$ be the $(d+1)$-permutation hypergraph on $[(d+1)k]$ vertices with edges $\{i,k+\sigma_1(i),\ldots,dk+\sigma_d(i)\}$ for $1\leq i\leq k$.

We want to show that there exists $c_2>0$ such that $B_n([\sigma_1,\ldots,\sigma_d])\leq c_2^n n^{n\left(1-\frac{1}{d}\right)}$ for all $n\in\mathbb{Z}^+$.

Note that $H$ is in essence the hypergraph corresponding to the set partition $[\sigma_1,\ldots,\sigma_d]$; the edges correspond to blocks. We can formalize this in the following definition.
\begin{defn}
Let $\pi$ be a set partition of $[n]$. Then the \emph{hypergraph corresponding to $\pi$} is simply the $1$-regular hypergraph whose edges are exactly given by the blocks of $\pi$.
\end{defn}

Note that in the case of hypergraphs corresponding to set partitions, the notion of set partition avoidance is exactly the same as that of hypergraph avoidance. Take any set partition $\pi'$ on $[n]$ avoiding $\pi$, and let $G$ be the hypergraph on $n$ vertices corresponding to $\pi'$. Then by this observation $G$ must avoid $H$.

Given a positive integer $s$ (possibly depending on $n$), we may construct a new hypergraph $G'$ on $[s]$ as follows. First, we divide $[n]$ into $s$ intervals $I_1,\ldots,I_s$ (in increasing order) so that each has size $\left\lfloor\frac{n}{s}\right\rfloor$ or $\left\lceil\frac{n}{s}\right\rceil$ (the number of each depends on the value of $n$ modulo $s$). For each edge $E\in G$, we construct an edge $E'$ on the vertex set $[s]$ by the rule that $j\in E'$ if and only if $I_j$ contains at least one vertex of $G$. Finally, we remove duplicate edges to obtain $G'$.

For example, if $G$ is the hypergraph $\{1,4\}, \{2,5,6\}, \{3\}$ on $[6]$ and $s=2$, then $I_1=\{1,2,3\}$ and $I_2=\{4,5,6\}$, and $G'$ will be on the vertex set $[2]$ and have edges $\{1,2\}$ and $\{1\}$.

Suppose that $G'$ contained $H$. Then we can find $k$ edges $E'_1,\ldots,E'_k$ in $G'$, and for each edge $E'_i$, $d+1$ vertices, $v'_{i,1},\ldots,v'_{i,d+1}$, that give the containment. But each edge $E'_i$ must arise from at least one $E_i\in G$. Choose such an $E_i$ for each $E'_i$. Then every vertex $v'_{i,j}$ must have at least one corresponding $v_{i,j}\in I_{v'_{i,j}}\cap E_i$, by the definition of $G$. Choose such a $v_{i,j}$ for every $v'_{i,j}$. Then the edges $E_i$ and the vertices $v_{i,j}$ represent a copy of $H$ in $G$, as the $v_{i,j}$ have relative ordering the same as that of $v'_{i,j}$ since $I_1,\ldots,I_s$ are arranged in increasing order. This contradicts our assumption that $G$ does not contain $H$, so $G'$ must in fact not contain $H$ either.

Note that $G'$ need not be $1$-regular, as in the example above, so we will begin by bounding the total number of hypergraphs on $s$ vertices avoiding $H$.

\begin{lemma}\label{genhyplemma}
There exists $c>0$ such that for all $n\in\mathbb{Z}^+$, there are at most $c^{n^d}$ ordered hypergraphs on $[n]$ that avoid $H$.
\end{lemma}
\begin{proof}[Proof of Lemma \ref{genhyplemma}]
We know that $i(G)=O(n^d)$  for any $G$ on $[n]$ avoiding $H$ by Theorem \ref{hypthm}. The exact same argument as in the proof of Theorem 2.5 from \cite{K2} (with $2$ replaced by $d+1$) then shows that the number of hypergraphs on $[n]$ avoiding $H$ is $2^{O(n^d)}$, as desired.
\end{proof}
Thus, there are at most $c^{s^d}$ possibilities for $G'$.

We now bound the number of set partitions $\pi$, and corresponding hypergraphs $G$ on $[n]$ that can correspond to a given $G'$ on $[s]$. It is clear that $i(G)\geq i(G')$ for any $G$ on $[n]$ corresponding to $G'$ on $[s]$ (as $G'$ is formed by contracting parts of edges and deleting duplicates), so since $i(G)=n$ (as $G$ corresponds to a set partition), $i(G')\leq n$.

The number of blocks of each size of $\pi$ correspond to an integer partition of $n$, and it is well known that there are $e^{o(n)}$ integer partitions of $n$. Now, fix an integer partition of $n$, and suppose $i$ occurs $c_i$ times. We want to bound the number of partition-correspondent hypergraphs $G$ with $c_i$ edges of size $i$ that correspond to $G'$. By counting vertices we see that $\displaystyle\sum_{i=1}^n ic_i=n$.

Each edge $E$ of size $i$ of $G$ corresponds to some edge $E'$ of size at most $i$ of $G'$. By a weak bounding argument, there are at most $n$ edges of $G'$, as $i(G')\leq n$. Once one of these at most $n$ edges is chosen to be $E'$, of size at most $i$, this gives $i$ size-$\frac{n}{s}$ intervals where the vertices of $E$ can lie. Thus, there are at most $\binom{\frac{in}{s}}{i}$ choices for $E$ given $E'$, giving at most $n\binom{\frac{in}{s}}{i}$ total choices for $E$. Choosing all size-$i$ edges simultaneously, and dividing by $c_i!$ to account for the fact that the edges are not distinguishable, we obtain that there are at most $\frac{\left(n\binom{\frac{in}{s}}{i}\right)^{c_i}}{c_i!}$ ways to choose all edges of size $i$ simultaneously. (Some choices of edges contradict each other--for example, if they share a vertex of $G$--but this will only decrease the number of options.) Therefore, the total number of ways to choose the set partition $\pi$ to correspond to $G'$ is at most
\[e^{o(n)}\displaystyle\max_{c_1+2c_2+\cdots+nc_n=n}\displaystyle\prod_{i=1}^n\frac{\left(n\binom{\frac{in}{s}}{i}\right)^{c_i}}{c_i!}.\]
Since there are at most $c^{s^d}$ ways to choose $G'$, this implies that
\[B_n(\pi)\leq c^{s^d}e^{o(n)}\displaystyle\max_{c_1+2c_2+\cdots+nc_n=n}\displaystyle\prod_{i=1}^n\frac{\left(n\binom{\frac{in}{s}}{i}\right)^{c_i}}{c_i!}.\]
By a (very) weak form of Stirling Approximation, $i!>\frac{i^i}{e^i}$ for all $i$ (for example, using $\left(1+\frac{1}{i}\right)^i<e$ and telescoping the left hand product from $i=1$ to $i=n-1$). Therefore,
\[\binom{\frac{in}{s}}{i}\leq\frac{\left(\frac{in}{s}\right)^i}{i!}=\frac{i^i}{i!}\left(\frac{n}{s}\right)^i<\left(\frac{en}{s}\right)^i.\]
Thus
\begin{align*}
\displaystyle\max_{c_1+2c_2+\cdots+nc_n=n}\displaystyle\prod_{i=1}^n\frac{\left(n\binom{\frac{in}{s}}{i}\right)^{c_i}}{c_i!} & \leq\displaystyle\max_{c_1+2c_2+\cdots+nc_n=n}\displaystyle\prod_{i=1}^n\frac{\left(n\left(\frac{en}{s}\right)^i\right)^{c_i}}{c_i!} \\ & \leq\displaystyle\max_{c_1+2c_2+\cdots+nc_n=n}\displaystyle\prod_{i=1}^n\frac{n^{c_i}\left(\frac{en}{s}\right)^{ic_i}}{c_i!} \\ & =\left(\frac{en}{s}\right)^n\displaystyle\max_{c_1+2c_2+\cdots+nc_n=n}\displaystyle\prod_{i=1}^n\frac{n^{c_i}}{c_i!}.
\end{align*}
Substituting this into our bound for $B_n(\pi)$, we obtain
\begin{align*}
B_n(\pi) & \leq c^{s^d} e^{o(n)}\left(\frac{en}{s}\right)^n \displaystyle\max_{c_1+2c_2+\cdots+nc_n=n}\displaystyle\prod_{i=1}^n\frac{n^{c_i}}{c_i!}\\ & =\frac{c^{s^d}}{s^n}e^{O(n)}n^n\displaystyle\max_{c_1+2c_2+\cdots+nc_n=n}\displaystyle\prod_{i=1}^n\frac{n^{c_i}}{c_i!}.
\end{align*}
The fraction on the left is the only part of this expression that depends on $s$, so we may choose $s$ to minimize it. The minimum occurs when $s$ is within a constant factor of $n^{\frac{1}{d}}$, so since we do not know the value of $c$, we will simply choose $s=n^{\frac{1}{d}}$. (Not coincidentally, this minimization is analagous to Brightwell's in \cite{B}, with the same result of $s\approx n^{\frac{1}{d}}$.)

Substituting this value of $s$, we see that $c^{s^d}=e^{O(n)}$, so we obtain the bound
\[B_n(\pi)\leq e^{O(n)}n^{n\left(1-\frac{1}{d}\right)}\displaystyle\max_{c_1+2c_2+\cdots+nc_n=n}\displaystyle\prod_{i=1}^n\frac{n^{c_i}}{c_i!}.\]
Therefore, to finish the problem and show that the right hand side is within an exponential factor of $n^{n\left(1-\frac{1}{d}\right)}$, it simply suffices to show that $\displaystyle\max_{c_1+2c_2+\cdots+nc_n=n}\displaystyle\prod_{i=1}^n\frac{n^{c_i}}{c_i!}=e^{O(n)}$, or in other words, that $\displaystyle\max_{c_1+2c_2+\cdots+nc_n=n}\displaystyle\sum_{i=1}^n (c_i\ln(n)-\ln(c_i!))=O(n)$.

By a previous approximation we know that $\ln(c_i!)>c_i\log c_i-c_i$, so
\[\max_{c_1+2c_2+\cdots+nc_n=n}\displaystyle\sum_{i=1}^n (c_i\ln(n)-\ln(c_i!))\leq \max_{c_1+2c_2+\cdots+nc_n=n}\displaystyle\sum_{i=1}^n (c_i(\ln(n)+1)-c_i\ln(c_i)).\]
Now, the function $c_i(\ln n+1)-c_i\ln c_i$ is concave in $c_i$, so we may use Lagrange Multipliers to maximize our expression subject to the restriction $\displaystyle\sum_{i=1}^n ic_i=n$. (The extrema of the domain, where all but one $c_i$ is $0$, clearly satisfy the desired inequality.) We see that the optimum occurs where the vectors $(\log n-\log c_i)$ and $(i)$ are proportional; that is, $c_i=\frac{n}{a^i}$ for some $a>0$ ($a$ will depend on $n$). Then
\[\displaystyle\sum_{i=1}^n(c_i(\ln(n)+1)-c_i\ln(c_i))=\displaystyle\sum_{i=1}^n (c_i+ic_i\ln(a))\leq n\ln(a)+n,\]
as $\displaystyle\sum_{i=1}^n c_i\leq\displaystyle\sum_{i=1}^n ic_i=n$. Thus, it suffices to show that $a$ is bounded independently of $n$.

The value of $a$ is determined by the equation $n=\displaystyle\sum_{i=1}^n ic_i=\displaystyle\sum_{i=1}^n\frac{in}{a^i}$, so $\displaystyle\sum_{i=1}^n\frac{i}{a^i}=1$. It is clear that $a>1$. Thus
\begin{align*}
1 & =\displaystyle\sum_{i=1}^n\frac{i}{a^i} \\ & <\displaystyle\sum_{i=1}^{\infty}\frac{i}{a^i} \\ & =\frac{\frac{1}{a}}{\left(1-\frac{1}{a}\right)^2} \\ & =\frac{a}{(a-1)^2}.
\end{align*}
Therefore, $(a-1)^2<a$, so $a^2-3a+1<0$, so $a<\frac{3+\sqrt{5}}{2}$. In particular, $a$ is bounded independently of $n$, finishing the proof.
\section{Proof of Theorem \ref{permthm}}
We now turn in the direction of parallel avoidance, by proving Theorem \ref{permthm}.

Note that the upper bound follows quite simply from Theorem \ref{mainthm}, as given $\sigma_1,\ldots,\sigma_d\in S_m$, each $(\sigma_1',\ldots,\sigma_d')\in S_n^d$ that avoids $(\sigma_1,\ldots,\sigma_d)$ yields a different set partition $[\sigma_1',\ldots,\sigma_d']$ of $[(d+1)n]$ avoiding $[\sigma_1,\ldots,\sigma_d]$, which has permutability $d$. Thus
\[S_n^d(\sigma_1,\ldots,\sigma_d)\leq B_{(d+1)n}([\sigma_1,\ldots,\sigma_d])\leq c_2^n ((d+1)n)^{\left(1-\frac{1}{d}\right)(d+1)n}\]
for some $c_2>0$, which gives the desired upper bound.

Now, we show the lower bound, which will turn out to follow easily from previously known results on random orders. Let $\sigma_1,\ldots,\sigma_d\in S_m$ with $m>1$. Then restricting $\sigma_i$ to their first two elements will yield some permutation that is an element of $S_2$; that is, either $12$ or $21$. Thus
\[S_n^d(\sigma_1,\ldots,\sigma_d)\geq S_n^d(\sigma'_1,\ldots,\sigma'_d)\]
where $\sigma'_i$ is either $12$ or $21$ for all $i$. Now, some permutation $\pi$ contains $21$ (that is, has an inversion) in exactly the indices where the complement of $\pi$ (if $\pi\in S_n$, the complement of $\pi$ is given by replacing each $i$ by $n+1-i$) contains $12$. Thus, replacing, say, all first permutations by their complement gives a bijection between $S_n^d(12,\sigma'_2,\ldots,\sigma'_d)$ and $S_n^d(21,\sigma'_2,\ldots,\sigma'_d)$. Doing this for all indices, we see that we can replace each $21$ by a $12$, so
\[S_n^d(\sigma_1,\ldots,\sigma_d)\geq S_n^d(12,\ldots,12).\]
Thus, it suffices to prove the lower bound when all permutations are $12$; that is, it suffices to show that there exists $c_1>0$ with
\[c_1^nn^{n\frac{d^2-1}{d}}\leq S_n^d(12,\ldots,12).\]

We first translate to the language of probabilities. Let $q_{d}(n)$ be the probability that randomly chosen $\sigma_1,\ldots,\sigma_k\in S_n$ will have $(\sigma_1,\ldots,\sigma_d)$ avoiding $(12,\ldots,12)$. Note that since there are $n!^k$ ways to choose $k$ permutations in $S_n$, $q_{d}(n)=\frac{S_n^d(12,\ldots,12)}{n!^d}$. We know that $n!$ is within an exponential factor of $n^n$ by Stirling approximation, so if we divide the desired statement by $n!^d$, we obtain that we want to show
\[c_1^n n^{-\frac{n}{d}}\leq q_{d}(n)\]
for all $n\in\mathbb{N}$ for some constant $c_1>0$.

We now translate the problem into the language of random $(d+1)$-dimensional orderings as follows. Let $p_1,\ldots,p_n$ be random points (in the usual sense) in $[0,1]^{d+1}$. We can sort them by their first coordinates. Once this is done, looking at the ordering of the $i^{th}$ coordinates of all $n$ points for some fixed $2\leq i\leq d+1$ will generate a permutation, so we get $d$ permutations $\sigma_1,\ldots,\sigma_d$ given by these orderings. It is easy to see that these permutations are independently and uniformly randomly chosen.

Now, we consider the (random) poset, also known as the random $(d+1)$-dimensional order $P_{d+1}(n)$, on these points as follows. We say that $p_i<p_j$ if and only if all coordinates of $p_i$ are less than those of $p_j$. Suppose $p_i$ has the $a_i$-th smallest first coordinate, and similarly $p_j$ has the $a_j$-th smallest. Then the condition that $p_i<p_j$ corresponds to (looking at the first coordinate) the condition that $a_i<a_j$, and (looking at the other $k$ coordinates) the condition that $\sigma_{\ell}(a_i)<\sigma_{\ell}(a_j)$ for all $\ell\in[k]$. This idea of relating sets of $d$ permutations to random $(d+1)$-dimensional orderings seems to go back to Winkler. \cite{W}

By definition, $(\sigma_1,\ldots,\sigma_d)$ avoids $(12,\ldots,12)$ if and only if there is no $b_1<b_2$ with $\sigma_i(b_1)<\sigma_i(b_2)$ for all $i$, and we can see by the previous paragraph that this is in turn equivalent to there being no pair of comparable elements in $P_{d+1}(n)$; that is, $P_{d+1}(n)$ is an antichain. Crane and Georgiou \cite{CG} derived from results of Brightwell \cite{B} that this probability is at least $\left(\frac{1}{e}+o(1)\right)^n n^{-\frac{n}{d}}$, finishing the proof of the lower bound.
\section{Set Partition Pattern Classes}
We may consider pattern classes of partitions; that is, collections of set partitions that are closed downward under containment, in the following sense.
\begin{defn}
A \emph{pattern class} $\mathcal{C}$ of set partitions is a set of set partitions such that if $A\in \mathcal{C}$ and $A$ contains $B$, $B\in \mathcal{C}$.
\end{defn}
For example, if $\pi$ is a set partition, the set partitions that avoid $\pi$ form a pattern class. For a pattern class $\mathcal{C}$, we can let $\mathcal{C}_n\subset\mathcal{C}$ consist of the partitions of $[n]$ in $\mathcal{C}$. We then can consider the growth rate of $C$ by looking at the sequence $|\mathcal{C}_n|$.

Theorem \ref{mainthm} allows us to prove the following result, classifying the growth rate of set partition pattern classes to within an exponential factor.
\begin{cor}
Let $\mathcal{C}$ be a nonempty pattern class of set partitions, not containing all set partitions, and let $d$ be the smallest positive integer such that there exists a set partition of permutability $d$ not in $\mathcal{C}$. Then there exists $c_2>c_1>0$ such that for all $n\in\mathbb{Z}^+$,
\[c_1^n n^{n\left(1-\frac{1}{d}\right)}\leq |\mathcal{C}_n|\leq c_2^n n^{n\left(1-\frac{1}{d}\right)}.\]
\end{cor}
\begin{proof}
Let $\pi$ have permutability $d$, and not be contained in $\mathcal{C}$. Then all elements of $\mathcal{C}$ must avoid $\pi$, by the definition of pattern class. Thus, $|\mathcal{C}_n|\leq B_n(\pi)$, so Theorem \ref{mainthm} proves the upper bound.

For the lower bound, it suffices to notice that the argument of Theorem \ref{lowerbound}gives in fact a lower bound on the number of partitions of $[n]$ of permutability at most $d-1$, and by our assumption all of these partitions are contained in $\mathcal{C}$. This proves the corollary.
\end{proof}

In particular, we have the following corollary.
\begin{cor}\label{classcor}
All pattern classes must grow as $B_n$, eventually become $0$, or grow within an exponential factor of $n^n{\left(1-\frac{1}{d}\right)}$ for some positive integer $d$.
\end{cor}
Corollary \ref{classcor} shows that there is an infinite sequence of `jumps' in pattern class growth rates, from $n^{n\left(1-\frac{1}{k-1}\right)}$ to $n^{n\left(1-\frac{1}{k}\right)}$ (modulo an exponential factor) for all $k$.
\section{Further Directions}
There are several possible directions to attempt to extend these results.

One such direction is the computation of the correct exponential factor for parallel avoidance in simple cases. That is, does $\displaystyle\lim_{n\to\infty}\sqrt[n]{\frac{S_n^d(\sigma_1,\ldots,\sigma_d)}{(n!)^{\frac{d^2-1}{d}}}}$ exist, and if so what is its value in simple cases? From results in \cite{CG} on $Q_3(n)$ we can see that in the $d=2$ case with $\sigma_1=\sigma_2=12$, we have a lower bound of $1$ and an upper bound of $3\sqrt{3(3\log 3-4\log 2)}\approx 3.76$. Weaker results in \cite{CG} apply for more than two copies of $12$, but no other cases appear to have been studied, and the problem appears quite difficult even in these cases. Similarly, we may try to compute the correct exponential factor for set partition avoidance in simple cases, though this seems similarly difficult. Another problem may be to prove the existence of the limit above, and similarly for set partitions, although the fact that this problem is open for the case of one permutation is not encouraging.

Another natural (perhaps more tractable) question is whether it is possible to classify the growth rates of pattern classes of $d$-tuples of permutations in a similar way to this paper's treatment of set partitions. When the pattern class has no basis elements, the answer is obviously $n!^d$, and with exactly one basis element, Theorem \ref{permthm} shows the speed of the pattern class is within exponential of $\left(n!\right)^{\frac{d^2-1}{d}}$. However, not all (proper) pattern classes grow at this rate; the class given by avoiding $(21,12)$ and $(21,21)$ simply grows as $n!$, as the first element of any pair in this pattern class must be the identity. Indeed, the product of a pattern class of $d$ tuples and a pattern class of $d'$-tuples will be a pattern class of $d+d'$-tuples, and using this, for any $d$, we may form a pattern class of $d$-tuples that grows within exponentially of $n^{\alpha n}$, where
\[\alpha=d-\displaystyle\sum_{i=1}^m\frac{1}{d_i},\]
with $m\in\mathbb{Z}^+$ and $\sum d_i\leq d$. Are any other growth rates possible?

Set partitions may be thought of as ordered graphs in which every connected component is a clique (the blocks are just given by the sets of vertices in connected components). Note that in this setting, set partition containment becomes the relation of taking an induced subgraph. Since all induced subgraphs also have all connected components cliques, they also correspond to set partitions. Thus all pattern classes of set partitions correspond to hereditary properties (properties closed under taking an induced subgraph) of ordered graphs. Thus Theorem \ref{mainthm} motivates the question: What may be said about factorial growth rates of hereditary properties of ordered graphs? The first superexponential jump (from $c^n$ to $c^n n^{\frac{n}{2}}$) is conjectured and proven in special cases in \cite{BBM}, but this problem still appears to be open, as well as that of higher jumps (such as those that exist in the set partition case, as given by Corollary \ref{classcor}).

A more tangential potential notion for further study is that of the permutability statistic and its distribution. To the authors' knowledge, this statistic has not explicitly appeared before in the literature, and given its strong connection to asymptotics, it may be worthwhile to study.
\section{Acknowledgements}
The first author would like to thank Ryan Alweiss for inspiring his interest in the problem of the asymptotics of set partition pattern avoidance, Adam Hammett for extremely helpful correspondence about the link to random $k$-dimensional orders and previous work done in that area, Mitchell Lee for conversations about parallel avoidance, and Adam Marcus and Mitchell Keller for their insight into their previous work.

\end{document}